\documentclass[a4paper,12pt]{article}
\usepackage{graphicx} 
\usepackage[T2A]{fontenc}
\usepackage[utf8]{inputenc}
\usepackage{dsfont}
\usepackage{amsfonts}
\usepackage{amssymb}
\usepackage{amsmath}
\usepackage{amsthm}
\usepackage{comment}
\usepackage{graphicx}
\usepackage{todonotes}
\usepackage{tikzsymbols}
\usepackage{enumitem}
\usepackage{authblk}
\usepackage{hyperref}

\usepackage{bm}

\usepackage[most]{tcolorbox}

\makeatletter
\DeclareRobustCommand{\cev}[1]{%
  \mathpalette\do@cev{#1}%
}
\newcommand{\do@cev}[2]{%
  \fix@cev{#1}{+}%
  \reflectbox{$\m@th#1\vec{\reflectbox{$\fix@cev{#1}{-}\m@th#1#2\fix@cev{#1}{+}$}}$}%
  \fix@cev{#1}{-}%
}
\newcommand{\fix@cev}[2]{%
  \ifx#1\displaystyle
    \mkern#23mu
  \else
    \ifx#1\textstyle
      \mkern#23mu
    \else
      \ifx#1\scriptstyle
        \mkern#22mu
      \else
        \mkern#22mu
      \fi
    \fi
  \fi
}

\newtheorem{thm}{Theorem}
\newtheorem*{thm*}{Theorem}
\newtheorem{lem}[thm]{Lemma}
\newtheorem{corol}[thm]{Corollary}

\newtheorem{defn}[thm]{Definition}
\newtheorem*{defn*}{Definition}
\newtheorem{convention}{Convention}
\newtheorem{remark}{Remark}

\def\sep{\,:\,}
\def\sub{\subseteq}
\def\defeq{:=}

\def\bbq{\mathbb{Q}}
\def\bbr{\mathbb{R}}

\def\defeq{:=}

\date{}
\begin{document}

\author{Dmitry Gayfulin}
\author{Erez Nesharim}
\affil{{Department of Mathematics, Technion -- Israel Institute of Technology, Haifa.} \textit{gamak.57.msk@gmail.com}, \textit{ereznesh@gmail.com}}



\title{Every real number is a sum of two real numbers with diverging partial quotients}

\maketitle

\begin{abstract}
We show that every irrational number is a sum of two real numbers with diverging partial quotients. The proof is constructive. The key towards these results is an algorithm which was recently developed by Nikita Shulga \cite{shulga}, and our study of this algorithm is of independent interest.
\end{abstract}

\section{Introduction}

Define the set 
$$G \defeq \{\alpha\in\bbr\sep \alpha\in\bbq \textrm{ or } a_n(\alpha)\xrightarrow[n\to\infty]{}\infty\}\,,$$
where $a_n(\alpha)$ is the $n$th partial quotient in the continued fraction expansion of a real number $\alpha$, whenever its continued fraction expansion is 
longer than $n$. 
The set $G$ was first studied by Good \cite{good}, where it is proved that its Hausdorff dimension is
$1/2$. 
In this paper we resolve a question suggested by Shulga \cite[Problem 6.2]{shulga} and show that every real number is a sum of two real numbers in $G$.

\begin{thm}\label{thm:main}
    For every $\alpha\in\bbr$ there exist $\beta,\gamma\in G$ such that 
    \begin{equation}\label{decomposition}
    \alpha=\beta+\gamma\,.
    \end{equation}
\end{thm}

Decompositions of real numbers as a sum of two numbers satisfying a Diophantine property are abundant in the literature. Arguably, the most notable examples of this sort are the theorem of Hall \cite{hall}, where $G$ is replaced by 
$$F(a) \defeq \{\alpha\in\bbr\sep a_n(\alpha)\leq a \textrm{ for every $n\geq1$}\}\,,$$
for $a=4$, and the theorem of Erd\H{o}s \cite{erdos}, where $G$ is replaced by the set of Liouville numbers. Another Diophatine property was studied by Cusick \cite{cusick}. Let 
$$S(a) \defeq \{\alpha\in\bbr\sep a_n(\alpha)\geq a\textrm{ for every $n\geq1$}\}\,,$$
where $a\geq2$ is a fixed integer. Following Cusick's result for $a=2$, Shulga \cite{shulga} proved that for any $a\geq2$ and any $0\leq\alpha\leq \frac{1}{a-1}$, there exists a decomposition \eqref{decomposition} with $\beta,\gamma\in S(a)$.

Our proof of Theorem \ref{thm:main} is constructive. In fact, our decomposition \eqref{decomposition} is given by an algorithm that was developed and analyzed by Shulga \cite{shulga}, and the proof relies on complementing its analysis. Shulga's algorithm is described in the next subsection, and the analysis we carry out is of independent interest. The reader is referred to Section \ref{prelim_sec} for more background on continued fractions. 


\subsection{Shulga's decomposition}
In \cite{shulga} Shulga established a nice and simple algorithm which decomposes real numbers into a sum of two real numbers with remarkable properties.

\begin{defn}[Shulga]\label{def:shulga}
For any $\alpha\in[0,1]$ define sequences $\{b_n\}_{n\geq1}$ and $\{c_n\}_{n\geq1}$, and real numbers $\beta$ and $\gamma$ as follows:
\begin{enumerate}[label=\arabic*.]
\item For every $n\geq0$, if the integers $b_k$ and $c_k$ are defined for every $1\leq k\leq n$, then if $$\alpha = [0;b_1,\ldots,b_{n}] + [0;c_1,\ldots,c_{n}]\,,$$ 
we put
\begin{equation*}
    \beta\defeq [0;b_1,\ldots,b_{n}]\,,\quad\textrm{and} \quad \gamma\defeq [0;c_1,\ldots,c_{n}]\,,
\end{equation*}
and finish the algorithm. Otherwise, define
\begin{equation*}
    \begin{aligned}
        b_{n+1}&\defeq a_{n+1}(\alpha-[0;c_1,\ldots,c_{n}])+1\,,\quad\textrm{and} \\
        c_{n+1}&\defeq a_{n+1}(\alpha-[0;b_1,\ldots,b_n,b_{n+1}])\,.
    \end{aligned}
\end{equation*}
\item If $b_n$ and $c_n$ are defined for every $n\geq1$, i.e., the algorithm takes infinitely many steps, then we put
\begin{equation*}
    \beta\defeq [0;b_1,b_2,\ldots]\,,\quad \textrm{and} \quad \gamma\defeq [0;c_1,c_2,\ldots]\,.
\end{equation*}
\end{enumerate}
\end{defn}

The following properties of the algorithm were established in \cite{shulga}:
\begin{thm}[Shulga]\label{thm:shulga}
For any $\alpha\in[0,1]$ the numbers $\beta$ and $\gamma$ in Definition \ref{def:shulga} are well defined and satisfy 
\begin{equation*}
    \alpha=\beta+\gamma\,.    
\end{equation*} 
Moreover, the sequences $\{b_n\}_{n\geq1}$ and $\{c_n\}_{n\geq1}$ in Definition \ref{def:shulga} satisfy 
\begin{equation*}
\begin{aligned}
    b_k &= a_k(\alpha-[0;c_1,\ldots,c_n])\,\quad\textrm{and} \\
    c_k &= a_k(\alpha-[0;b_1,\ldots,b_n])\,,
\end{aligned}
\end{equation*}
for any $1\leq k\leq n$.
\end{thm}

As will be seen below, a key step in Shulga's proof of Theorem \ref{thm:shulga} is to show that sequences $\{b_n\}_{n\geq1}$ and $\{c_n\}_{n\geq1}$ of Definition \ref{def:shulga} satisfy the inequality (see Section \ref{sec_intersec})
\begin{equation}\label{eq:main-shulga}
    c_n\ge b_n\,,
\end{equation}
for every $n\geq1$. Likewise, a key step in the proof of Theorem
\ref{thm:main} is a strengthening of 
\eqref{eq:main-shulga}:

\begin{thm}
\label{Th1}
For any $\alpha\in[0,1]$ and $n>1$ the Shulga decomposition 
satisfies 
\begin{equation}
\label{cnbn_main_ineq}
c_{n} > b_n\,.
\end{equation}
\end{thm}

In Section \ref{sec:rationals} we will show that Theorem \ref{Th1} implies the following corollary, which implies Theorem \ref{thm:main}:
\begin{corol}\label{cor:main}
    For any $\alpha\in[0,1]$ the Shulga decomposition \eqref{decomposition} satisfies $\beta,\gamma\in G$. Moreover, if $\alpha\in\bbq$ then $\beta,\gamma\in \bbq$.
\end{corol}

\begin{remark}
If $\alpha=p/q\in\bbq$, one can ask what is the length of the continued fraction expansions of the $\beta$ and $\gamma$ in the Shulga decomposition. It follows from the proof of Corollary \ref{cor:main} that this length is $O\left(q^2\right)$. 
Numerical evidence seem to suggest that the maximal length is $O(\log{q})$, but the study of this phenomenon lies beyond the scope of the current paper.
\end{remark}

The key towards Theorem \ref{Th1} is a detailed analysis of the sets of reals with a given first $n$ steps of Shulga's algorithm. This analysis allows us to prove further properties regarding the growth rate of the partial quotients of $\beta$ and $\gamma$. 

\begin{thm}\label{Th2}
For any $\alpha\in[0,1]$ and $n\geq1$ the Shulga decomposition 
satisfies 
\begin{align}
\label{monot_main_ineq}
b_{n+1}&\ge b_n\,,\\
\label{monot_main_ineq2}
b_{n+2}&\ge b_n+1\,,\\
\label{bn_main_ineq}
b_n&\ge n\,.
\end{align}
\end{thm}


The estimates in Theorems \ref{Th1} and 
\ref{Th2}
are non-improvable, as one can see from the Shulga decomposition 
\begin{equation*}
    531/629 = [0; 2, 2, 3, 5] + [0; 2, 3, 4, 8]\,,
\end{equation*}
and the inequalities \eqref{monot_main_ineq} and \eqref{monot_main_ineq2} have no analogues for the sequence $\{c_n\}_{n\geq1}$. For example, the Shulga decomposition
\begin{equation*}
28244/141973=[0;6,27]+[0;30,29]\,,    
\end{equation*}
shows that $\{c_n\}_{n\geq1}$ is not necessarily monotonic. This gives a negative answer to \cite[Problem 6.1]{shulga}. In fact, one can show that for every $l > 1$, there exists $\alpha$ such that $c_{l} < c_1$. Here is an example with $l=4$:
\begin{equation*}
9974074083712426/149649898029019789 = [0;16, 227, 231, 235]+[0;240, 229, 233, 237]\,.    
\end{equation*}
Numerical evidence suggest that \eqref{bn_main_ineq} can be improved for large $n$, and more specifically, that for every $\alpha$ the Shulga decomposition should satisfy $b_n> 2n$ for all $n$ large enough. However, the following theorem shows that a linear growth rate is best possible. 
\begin{thm}
\label{examp_thm}
There exists an irrational number such that the partial quotients of its Shulga decomposition \eqref{decomposition} satisfy the inequalities
\begin{equation}
\label{2n_3n}
4n-2\le b_{n} < c_n < 5n.
\end{equation}
\end{thm}
Theorem \ref{examp_thm} is proved by an explicit construction. 

\begin{remark}
    It is possible to deduce that the number constructed in the proof of Theorem \ref{examp_thm} satisfies $c_n < (4+\varepsilon)n$ for every $\varepsilon>0$ and all $n$ large enough. However, we decided to settle with the estimate \eqref{2n_3n} in order to simplify the proof.
\end{remark}


\section{Basic properties of continued fractions}
\label{prelim_sec}
In this section, we recall some basic
facts about continued fractions that we use in the article. For a more detailed background, we refer to the classical literature on the subject (e.g., \cite{khinchin, rock_sz,schmidt}). Any real number $\alpha$ can be represented as a continued fraction as follows
\begin{equation*}
\alpha=[a_0;a_1,a_2,\ldots]\defeq a_0+ \cfrac{1}{a_1+\cfrac{1}{a_2+\cdots}}\,,
\end{equation*}
where $a_0\in\mathbb{Z}$ and $a_i\in\mathbb{Z}_{+}$ for $i\ge 1$. This representation is finite for rational $\alpha$, and infinite and unique for irrational $\alpha$. Each rational number has exactly two different representations
\begin{equation*}
[a_0;a_1,a_2,\ldots,a_n]=[a_0;a_1,a_2,\ldots,a_n-1,1]\,,\quad a_n\ge 2\,.
\end{equation*}
Thus, without loss of generality we will assume that the last partial quotient of rational numbers is greater than $1$. 

For any $n\ge1$ we consider $a_n(\cdot)$ as a function $\mathbb{R}\to\mathbb{N}$ that returns the $n$th partial quotient of the continued fraction expansion. Of course, if $\alpha$ is rational, then $n$ should not exceed the length of the continued fraction expansion of $\alpha$; otherwise, $a_n(\alpha)$ is undefined.

The convergents of a continued fraction are the rational numbers 
$$
\frac{p_n(\alpha)}{q_n(\alpha)}:=[a_0;a_1,a_2,\ldots,a_n]
$$
whose numerators and denominators are defined by the recurrence relations
\begin{equation}\label{eq:convergents}
\begin{aligned}
    p_0(\alpha)&=a_0\,,\; &p_1(\alpha)&=a_1a_0+1\,,\;
    &p_{n}(\alpha)&=a_{n}p_{n-1}(\alpha)+p_{n-2}(\alpha)\,,\quad\textrm{and}\\
    q_0(\alpha)&=1\,,\;&q_1(\alpha)&=a_1\,,\;
    &q_{n}(\alpha)&=a_{n}q_{n-1}(\alpha)+q_{n-2}(\alpha)\,.
\end{aligned}
\end{equation}
The following properties are well known
\begin{equation}\label{eq:common1}
    q_n(\alpha)p_{n-1}(\alpha)-p_n(\alpha)q_{n-1}(\alpha) = (-1)^n\,,
\end{equation}
\begin{equation}\label{eq:common}
    \frac{q_n(\alpha)}{q_{n-1}(\alpha)} = [a_n;a_{n-1},\ldots,a_{1}]\,.
\end{equation}

It is well known that the sequence of even convergents $\frac{p_{2n}(\alpha)}{q_{2n}(\alpha)}$ is monotonically increasing, converging to $\alpha$ from the left, while the sequence of odd convergents $\frac{p_{2n-1}(\alpha)}{q_{2n-1}(\alpha)}$ is monotonically decreasing, converging to $\alpha$ from the right. In our argument, we will frequently use intervals with endpoints of the form $\frac{p_n}{q_n}$. The order of those endpoints will depend on the parity of $n$. To avoid considering multiple cases, we will use the following convention:
\begin{convention}\label{conv}
    For any $x,y,z\in\bbr$, the \emph{half-open-half-closed interval whose endpoints are $x+z$ and $y+z$} is 
    $$
    [x,y)+z \defeq \begin{cases}
        \{\alpha\in\bbr\sep x+z\leq\alpha < y+z\} &\text{if $x < y$},\\
        \{\alpha\in\bbr\sep y+z<\alpha \leq x+z\} &\text{if $x > y$}.
    \end{cases}
    $$
\end{convention}
Let $a_1,a_2,\ldots, a_n$ be arbitrary positive integers. Then the set of real numbers whose continued fraction expansion begins with $a_1,\ldots,a_n$ is the interval
\begin{equation}\label{cyl_def}
\left\{\alpha\in [0,1] \sep a_k(\alpha)=a_k \textrm{ for every $1\leq k\leq n$}\right\} = 
\bigl[[0;a_1,\ldots,a_n]\,, [0;a_1,\ldots,a_n+1]\bigr)\,.
\end{equation}
Note that 
\begin{equation*}\label{eq:order}
\begin{aligned}
\relax
    [0;a_1,\ldots,a_n]&<[0;a_1,\ldots,a_n+1]\,,\quad\textrm{for even $n$, and} \\
    [0;a_1,\ldots,a_n]&>[0;a_1,\ldots,a_n+1]\,,\quad\textrm{for odd $n$},
\end{aligned}
\end{equation*}
so for any $n$ odd Convention \ref{conv} is applied in \eqref{cyl_def}. Lastly, recall that for any $a_1,\ldots,a_{n-1}$ and $a$ integers, the length $n$ continued fraction $[0;a_1,\ldots,a_{n-1},a]$ converges monotonically to $[0;a_1,\ldots,a_{n-1}]$ as $a$ grows to infinity. In fact,
\begin{equation}\label{eq:mon}
\begin{aligned}
\relax
    [0;a_1,\ldots,a_{n-1},a]&\underset{a\to\infty}{\nearrow}[0;a_1,\ldots,a_{n-1}]\,,\textrm{ for even $n$, and} \\[2pt]
    [0;a_1,\ldots,a_{n-1},a]&\underset{a\to\infty}{\searrow}[0;a_1,\ldots,a_{n-1}]\,,\textrm{ for odd $n$}.
\end{aligned}
\end{equation}

\section{Shulga intervals and their intersections}
\label{sec_intersec}
In this section we rephrase the Shulga algorithm in Definition \ref{def:shulga} in terms of intersection of certain intervals. Analyzing the arrangements of these intervals is the key for many of our results. Along the way we will rephrase some results and proofs from \cite{shulga}. In particular, the analysis that appears in this section does not rely on Theorem \ref{thm:shulga} and offers a self-contained proof of it.

For every $n\geq1$ and positive integers $b_1,\ldots,b_n\geq2$ and $c_1,\ldots,c_n$, define for every $1\leq k \leq n$ the intervals
\begin{equation}
\begin{aligned}
\label{cn_bn_def}
B_{k}&\defeq\bigl[[0;b_1,\ldots,b_{k}-1],[0;b_1,\ldots,b_{k}]\bigr)+[0;c_1,\ldots,c_{k-1}]\,,\\
C_{k}&\defeq\bigl[[0;c_1,\ldots,c_{k}],[0;c_1,\ldots,c_{k}+1]\bigr)+[0;b_1,\ldots,b_{k}]\,.
\end{aligned}
\end{equation}
Note that for odd $k$ the left and right endpoints of these intervals interchange according to our Convention \ref{conv}. Also, define
\begin{equation*}
    A_n\defeq\bigcap\limits_{k=1}^{n}(B_k\cap C_k)\,,
\end{equation*}
and set $A_0\defeq B_0\defeq C_0\defeq [0,1]$.
\begin{lem}
\label{large_intersec_lem}
For every $n\geq0$ and positive integers $b_1,\ldots,b_n\geq2$ and $c_1,\ldots,c_n$, if 
$
\alpha\in A_n
$
then the first $n$ steps in Definition \ref{def:shulga} are well defined and yield the two sequences $b_1,\ldots,b_{n}$ and $c_1,\ldots,c_{n}$.
\end{lem}

\begin{proof}
By induction on $0\leq k\leq n-1$. The base case $k=0$ is trivial. Assume that the first $k$ steps in Definition \ref{def:shulga} are $b_1,\ldots,b_{k}$ and $c_1,\ldots,c_{k}$. Since $\alpha\in B_{k+1}$ one has
\begin{equation*}
\alpha-[0;c_1,\ldots,c_{k}]\in\bigl[[0;b_1,\ldots,b_{k+1}-1],[0;b_1,\ldots,b_{k+1}] \bigr),
\end{equation*}
and hence by \eqref{cyl_def} the $(k+1)$th partial quotient of $\alpha-[0;c_1,\ldots,c_{k}]$ is well defined and satisfies
$$
b_{k+1}=a_{k+1}(\alpha-[0;c_1,\ldots,c_{k}])+1\,.
$$ 
In particular, we have $\alpha-[0;c_1,\ldots,c_{k}]\neq[0;b_1,\ldots,b_{k}]$, and the $(k+1)$th digit of the first sequence in Definition \ref{def:shulga} is well defined and equals $b_{k+1}$. Similarly, since $\alpha\in C_{k+1}$, one has
\begin{equation*}
\alpha-[0;b_1,\ldots,b_{k+1}]\in\bigl[[0;c_1,\ldots,c_{k+1}],[0;c_1,\ldots,c_{k+1}+1]\bigr)
\end{equation*}
and hence by \eqref{cyl_def} the $(k+1)$th partial quotient of $\alpha-[0;b_1,\ldots,b_{k+1}]$ is well defined and satisfies
$$
c_{k+1}=a_{k+1}(\alpha-[0;b_1,\ldots,b_{k+1}])\,,$$
which again agrees with Definition \ref{def:shulga}. This finishes the induction.
\end{proof}

The following lemma plays a key role in the proofs of Theorems \ref{Th1} and \ref{Th2}. Only the last part of it is new and the rest were established by Shulga \cite{shulga}, who used equivalent formulations. However, in order to practice the usefulness of Lemma \ref{large_intersec_lem}, we will provide a self-contained proof for all five parts. 

\begin{lem}
\label{Cn_Bn_lem}
Suppose that $b_1,\ldots,b_n\geq2$ and $c_1,\ldots,c_n$ are positive integers, and let 
\begin{equation}
    \label{eq:fractions}
    \frac{p_k}{q_k} \defeq [0;b_1,\ldots,b_k] \,,\quad \textrm{and}\quad \frac{s_k}{t_k} \defeq [0;c_1,\ldots,c_k]
\end{equation}
be the numerator and denominator of the corresponding continued fraction. We have:
\begin{enumerate}[wide, labelindent=0pt]
    \item \label{item:cnbn}$B_n\cap C_n\ne\emptyset$ if and only if
        \begin{equation}
        \label{CnBn_ineq1}
        \frac{1}{t_{n-1}(t_n+t_{n-1})}<\frac{1}{q_n(q_n-q_{n-1})}\,.
        \end{equation}
    \item \label{item:cncn1}
    $C_n\cap C_{n-1}\ne\emptyset$ 
    if and only if \begin{equation}
        \label{CnCn1_ineq0}
        \frac{1}{q_{n}q_{n-1}}<\frac{c_{n}}{(t_{n}+t_{n-1})(t_{n-1}+t_{n-2})}\,.
    \end{equation}
    \item \label{item:bncn1_lem}
        $B_n\cap C_{n-1}\ne\emptyset$ if and only if
        \begin{equation*}\label{bncn1_ineq04}
            \frac{1}{q_nq_{n-1}}<\frac{1}{t_{n-1}(t_{n-1}+t_{n-2})}\,.
        \end{equation*}
    \item \label{item:cn_gt_lem}
        If $B_n\cap C_n\cap C_{n-1}\ne\emptyset$ then
        \begin{equation*}
        \label{CnCn1_ineq03}
                c_{n}>(1+[0;c_{n-1},\ldots,c_1])(b_{n}-1+[0;b_{n-1},\ldots,b_1])\,.
        \end{equation*}
    \item \label{bn_gt_lem}
    If $B_{n-1}\cap C_n\cap C_{n-1}\ne\emptyset$ then
    \begin{equation}\label{bn_ge_ineq_main}
        b_n>\left(1+\frac{1}{c_n}\right)\left(c_{n-1}+1+[0;c_{n-2},\ldots,c_1]\right)(1-[0;b_{n-1},\ldots,b_1])-1\,.
    \end{equation}
\end{enumerate}
\end{lem}

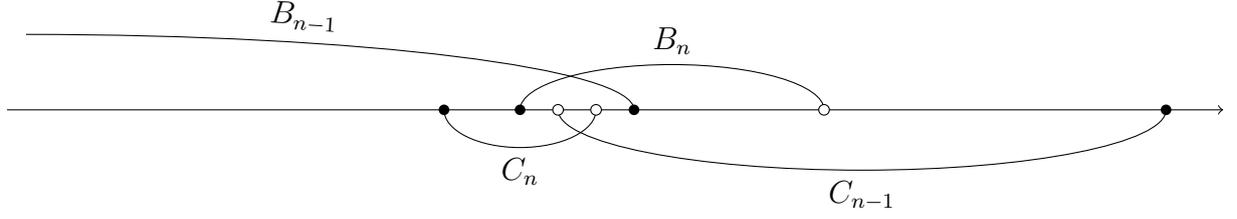
\begin{figure}
\centering
\scalebox{1}{
\begin{tikzpicture}
    \draw[->] (0,0) -- (16,0);
    \draw (5.75,0) arc (180:360:1 and 0.5) node[pos=.5,below,sloped] () {$C_{n}$};
    \fill (5.75,0) circle (0.07);
    \fill[white] (7.75,0) circle (0.07);
    \draw (7.75,0) circle (0.07);
    \draw (7.25,0) arc (180:360:4 and 0.8) node[pos=.5,below,sloped] () {$C_{n-1}$};
    \fill[white] (7.25,0) circle (0.07);
    \draw (7.25,0) circle (0.07);
    \fill (15.25,0) circle (0.07);
    \draw (10.75,0) arc (0:180:2 and 0.6) node[pos=.5,above,sloped] () {$B_{n}$};
    \fill[white] (10.75,0) circle (0.07);
    \draw (10.75,0) circle (0.07);
    \fill (6.75,0) circle (0.07);
    \draw (8.25,0) arc (0:90:8 and 1) node[pos=.7,above,sloped] () {$B_{n-1}$};
    \fill (8.25,0) circle (0.07);
\end{tikzpicture}
}
\caption{A possible arrangement of the intervals $B_{n-1},C_{n-1},B_n,C_n$ for $n$ even. An endpoint is filled or empty according to whether the endpoint is within the interval or not.}
\label{fig:intervals}
\end{figure}

\begin{proof}
Without loss of generality we assume that $n$ is even. Indeed, otherwise, the argument goes through by applying Convention \ref{conv} to all the intervals in the argument.
\begin{enumerate}[wide,labelindent=0pt]
    \item \label{item:cnbn_proof}Recall the definition \eqref{cn_bn_def} with $k=n$
    \begin{equation*}\label{eq:bncn}
        \begin{aligned}
            B_{n}&=\bigl[[0;b_1,\ldots,b_{n}-1], [0;b_1,\ldots,b_{n}]\bigr)+[0;c_1,\ldots,c_{n-1}]\,,\\
            C_{n}&=\bigl[[0;c_1,\ldots,c_{n}],[0;c_1,\ldots,c_{n}+1]\bigr)+[0;b_1,\ldots,b_{n}]\,.
        \end{aligned}
    \end{equation*}
    Since $[0;c_1,\ldots,c_{n}+1]<[0;c_1,\ldots,c_{n-1}]$, we have that $B_n\cap C_n\neq\emptyset$ is equivalent to the fact that the left endpoint of $B_n$ is less than the right endpoint of $C_n$ (see Figure \ref{fig:intervals}). In other words, to
    \begin{equation*}
        [0;b_1,\ldots,b_{n}-1]+[0;c_1,\ldots,c_{n-1}]< [0;b_1,\ldots,b_{n}]+[0;c_1,\ldots,c_{n}+1]\,.
    \end{equation*}
    This is equivalent to
    \begin{equation}
    \label{CnBn_ineq1_expl}
    [0;c_1,\ldots,c_{n-1}]-[0;c_1,\ldots,c_{n}+1]<[0;b_1,\ldots,b_{n}]-[0;b_1,\ldots,b_{n}-1]\,.
    \end{equation}
A direct computation shows that the definition \eqref{eq:fractions} implies that
\begin{equation}\label{eq:direct1}
\begin{aligned}
\relax
    [0;b_1,\ldots,b_{n}]-[0;b_1,\ldots,b_{n}-1]
    &=
    \frac{p_n}{q_n}-\frac{p_n-p_{n-1}}{q_n-q_{n-1}}
    \\[2pt]
    &=
    \frac{(q_n-q_{n-1})p_n - q_n(p_n-p_{n-1})}{q_n(q_n-q_{n-1})}
    \\[2pt]
    &=
    \frac{-q_{n-1}p_n + q_np_{n-1}}{q_n(q_n-q_{n-1})}
    \\[2pt]
    &=
    \frac{1}{q_n(q_n-q_{n-1})}\,,
\end{aligned}
\end{equation}
where the last equality holds by \eqref{eq:common1} since $n$ is even, and, similarly,
\begin{equation}\label{eq:direct2}
\begin{aligned}
\relax
    [0;c_1,\ldots,c_{n-1}]-[0;c_1,\ldots,c_{n}+1]
    &=
    \frac{s_{n-1}}{t_{n-1}}-\frac{s_n+s_{n-1}}{t_n+t_{n-1}}
    \\[2pt]
    &=
    \frac{(t_n+t_{n-1})s_{n-1} - t_{n-1}(s_n+s_{n-1})}{t_{n-1}(t_n+t_{n-1})}
    \\[2pt]
    &=
    \frac{t_{n}s_{n-1} - t_{n-1}s_{n}}{t_{n-1}(t_n+t_{n-1})}
    \\[2pt]
    &=
    \frac{1}{t_{n-1}(t_n+t_{n-1})}\,.
\end{aligned}
\end{equation}
Finally, equations \eqref{eq:direct1} and \eqref{eq:direct2} show that \eqref{CnBn_ineq1_expl} is equivalent to \eqref{CnBn_ineq1}.

\item Definition \eqref{cn_bn_def} with $k=n-1$ gives 
\begin{align*}
C_{n-1}&=\bigl[[0;c_1,\ldots,c_{n-1}], [0;c_1,\ldots,c_{n-1}+1]\bigr)+[0;b_1,\ldots,b_{n-1}]\,.
\end{align*}
Since $[0;c_1,\ldots,c_{n}+1]<[0;c_1,\ldots,c_{n-1}]$ and $[0;b_1,\ldots,b_{n}]<[0;b_1,\ldots,b_{n-1}]$, we have that $C_n\cap C_{n-1}\neq\varnothing$ is equivalent to (see Figure \ref{fig:intervals})
\begin{equation*}
[0;b_1,\ldots,b_{n-1}]+[0;c_1,\ldots,c_{n-1}+1]<[0;b_1,\ldots,b_{n}]+[0;c_1,\ldots,c_{n}+1]\,.
\end{equation*}
This inequality is equivalent to
\begin{equation}\label{CnCn1_ineq0_full}
[0;b_1,\ldots,b_{n-1}]-[0;b_1,\ldots,b_{n}]<[0;c_1,\ldots,c_{n}+1]-[0;c_1,\ldots,c_{n-1}+1]\,.
\end{equation}
As in the proof of Lemma \ref{Cn_Bn_lem}\eqref{item:cnbn}, a direct computation shows that
\begin{equation}\label{eq:direct3}
\begin{aligned}
\relax
[0;b_1,\ldots,b_{n-1}]-[0;b_1,\ldots,b_{n}] &= \frac{1}{q_{n}q_{n-1}}\,,
\end{aligned}
\end{equation}
and
\begin{equation}\label{eq:direct4}
\begin{aligned}
&
[0;c_1,\ldots,c_{n}+1]-[0;c_1,\ldots,c_{n-1}+1] = 
\\
&
[0;c_1,\ldots,c_{n-1}] - [0;c_1,\ldots,c_{n-1}+1] - ([0;c_1,\ldots,c_{n-1}] - [0;c_1,\ldots,c_{n}+1])=
\\[2pt]
&
\frac{1}{t_{n-1}(t_{n-1}+t_{n-2})}-\frac{1}{t_{n-1}(t_{n-1}+t_n)} =
\\[2pt]
&
\frac{1}{t_{n-1}}\left(\frac{t_{n}-t_{n-2}}{(t_{n}+t_{n-1})(t_{n-1}+t_{n-2})}\right)=
\frac{c_{n}}{(t_{n}+t_{n-1})(t_{n-1}+t_{n-2})}
\,.
\end{aligned}
\end{equation}
Finally, equations \eqref{eq:direct3} and \eqref{eq:direct4} show that \eqref{CnCn1_ineq0_full} is equivalent to \eqref{CnCn1_ineq0}.

\item We have that $B_{n}\cap C_{n-1}\ne\emptyset$ is equivalent to
\begin{equation*}
[0;b_1,\ldots,b_{n-1}]+[0;c_1,\ldots,c_{n-1}+1]<[0;b_1,\ldots,b_n]+[0;c_1,\ldots,c_{n-1}]\,,
\end{equation*}
which is equivalent to 
\begin{equation*}
\label{c_n_bn1}
\frac{1}{q_{n}q_{n-1}}<\frac{1}{t_{n-1}(t_{n-1}+t_{n-2})}\,.
\end{equation*}

\item Combining \eqref{CnBn_ineq1} and \eqref{CnCn1_ineq0} gives
\begin{equation}\label{CnCn1_ineq02}
\begin{aligned}
\frac{1}{t_{n-1}(t_{n}+t_{n-1})}\frac{q_{n}-q_{n-1}}{q_{n-1}} 
&<
\frac{1}{q_{n}(q_{n}-q_{n-1})}\frac{q_{n}-q_{n-1}}{q_{n-1}}
\\[2pt]
&=
\frac{1}{q_{n-1}q_{n}}<\frac{c_{n}}{(t_{n-1}+t_{n})(t_{n-1}+t_{n-2})}\,.
\end{aligned}
\end{equation}
By applying \eqref{eq:common}, the inequality \eqref{CnCn1_ineq02} is equivalent to
\begin{equation*}
\label{CnCn1_ineq03_full}
c_{n}>\frac{t_{n-1}+t_{n-2}}{t_{n-1}}\frac{q_{n}-q_{n-1}}{q_{n-1}}=(1+[0;c_{n-1},\ldots,c_1])(b_{n}-1+[0;b_{n-1},\ldots,b_1])\,.
\end{equation*}
\item Applying \ref{CnBn_ineq1} with $n-1$ instead of $n$ yields
\begin{equation}
\label{Cn_min1Bn_min1_ineq1}
\frac{1}{t_{n-2}(t_{n-1}+t_{n-2})} < \frac{1}{q_{n-1}(q_{n-1}-q_{n-2})}\,.
\end{equation}
Combining \eqref{CnCn1_ineq0_full} and \eqref{Cn_min1Bn_min1_ineq1} gives
\begin{equation*}
\begin{aligned}
\frac{1}{q_{n}q_{n-1}}<\frac{1}{t_{n-1}(t_{n-1}+t_{n-2})}\frac{c_nt_{n-1}}{t_n+t_{n-1}}&=
\frac{1}{t_{n-2}(t_{n-1}+t_{n-2})}\frac{t_{n-2}}{t_{n-1}}\frac{c_nt_{n-1}}{t_n+t_{n-1}}
\\[2pt]
&
<
\frac{1}{q_{n-1}(q_{n-1}-q_{n-2})}\frac{c_n t_{n-2}}{t_n+{t_{n-1}}}\,.
\end{aligned}
\end{equation*}
Hence, by \eqref{eq:common},
\begin{equation*}\label{bn_gt_interm_ineq}
\begin{aligned}
\frac{q_{n}}{q_{n-1}}\frac{q_{n-1}}{q_{n-1}-q_{n-2}}>\frac{t_{n}+t_{n-1}}{c_n{t_{n-2}}}&=\frac{(c_{n}+1)t_{n-1}+t_{n-2}}{c_n{t_{n-2}}}
\\[2pt]
&=
\left(1+\frac{1}{c_n}\right)[c_{n-1};c_{n-2},\ldots,c_1]+\frac{1}{c_n}\,.
\end{aligned}
\end{equation*}
Applying \eqref{eq:common} once again results in
\begin{equation*}
\begin{aligned}
\frac{[b_n;b_{n-1},\ldots,b_1]}{1-[0;b_{n-1},\ldots,b_1]}&>
\left(1+\frac{1}{c_n}\right)(c_{n-1}+1+[0;c_{n-2},\ldots,c_1])-1\,,
\end{aligned}
\end{equation*}
thus establishing the inequality \eqref{bn_ge_ineq_main}.
\end{enumerate}
\end{proof}

We are now ready to rephrase and prove Theorem \ref{thm:shulga}.

\begin{lem}
Given $\alpha\in[0,1]$, either there exist unique $n\geq0$ and finite sequences $b_1,\ldots,b_n$ and $c_1,\ldots,c_n$ such that $\alpha\in A_n$ and 
\begin{equation*}
\label{alpha_fin_decomp}
\alpha=[0;b_1,\ldots,b_n]+[0;c_1,\ldots,c_n]
\end{equation*}
or there exist unique infinite sequences $b_1,b_2,\ldots$ and $c_1,c_2,\ldots$ such that
$$
\alpha\in \bigcap_{n=0}^\infty A_n\,.
$$
\end{lem}

\begin{proof}
By induction on $n$. Let $n\ge 1$, $b_1,\ldots, b_n\geq2$ and $c_1,\ldots, c_n$ satisfy $\alpha\in A_n$. Recall that by definition \eqref{eq:bncn} one has
\begin{equation*}
    \begin{aligned}
        B_{n}&=\bigl[[0;b_1,\ldots,b_{n}-1], [0;b_1,\ldots,b_{n}]\bigr)+[0;c_1,\ldots,c_{n-1}]\,,\\
        C_{n}&=\bigl[[0;c_1,\ldots,c_{n}],[0;c_1,\ldots,c_{n}+1]\bigr)+[0;b_1,\ldots,b_{n}]\,,
    \end{aligned}
\end{equation*}
and for every two integers $b\geq2$ and $c\geq 1$ let this notation be extended by setting
\begin{align*}
B_{n+1}(b)&=\bigl[[0;b_1,\ldots,b_{n},b-1],[0;b_1,\ldots,b_{n},b]\bigr)+[0;c_1,\ldots,c_{n}]\,,\\
C_{n+1}(b,c)&=\bigl[[0;c_1,\ldots,c_{n},c],[0;c_1,\ldots,c_{n},c+1]\bigr)+[0;b_1,\ldots,b_{n},b]\,.
\end{align*}
It is enough to show that
\begin{align}
\label{subset_C_inB_n1}
C_n&\sub\bigcup\limits_{b=2}^{\infty}B_{n+1}(b)\cup\{[0;b_1,\ldots,b_{n}]+[0;c_1,\ldots,c_{n}]\}\,,\quad\textrm{and}\\
\label{subset_Bn1_inC_n1}
B_{n+1}(b)&\sub\bigcup\limits_{c=1}^{\infty}C_{n+1}(b,c)\,,\quad\textrm{for any $b$ such that $B_{n+1}(b)\cap C_n\neq\emptyset$}\,.
\end{align}
Since the unions in \eqref{subset_C_inB_n1} and \eqref{subset_Bn1_inC_n1} are disjoint, this also implies the uniqueness in the statement of the theorem. It follows from \eqref{eq:mon} that the right-hand sides of \eqref{subset_C_inB_n1} and \eqref{subset_Bn1_inC_n1} are the intervals
\begin{align*}
&\bigl[[0;b_1,\ldots,b_{n}], [0;b_1,\ldots,b_{n}+1]\bigr]+[0;c_1,\ldots,c_{n}]\,,\quad\textrm{and}\\
&\bigl([0;c_1,\ldots,c_{n}],[0;c_1,\ldots,c_{n}+1]\bigr]+[0;b_1,\ldots,b_{n},b]\,.
\end{align*}
respectively. Thus, the containments in \eqref{subset_C_inB_n1} and \eqref{subset_Bn1_inC_n1} are equivalent to the inequalities
\begin{align}
\label{eq:first}
[0;b_1,\ldots,b_n+1]+[0;c_1,\ldots,c_{n}]&\ge[0;b_1,\ldots,b_n]+[0;c_1,\ldots,c_{n}+1]\,,\quad\textrm{and}\\
\label{eq:second}
[0;b_1,\ldots,b_{n},b]+[0;c_1,\ldots,c_{n},1]&> [0;b_1,\ldots,b_{n},b-1]+[0;c_1,\ldots,c_{n}]\,.
\end{align}
Similarly to the way it was done in the proof of Lemma \ref{Cn_Bn_lem}, the inequality \eqref{eq:first} can rewritten as
\begin{equation}
\label{tn_le_qn}
\frac{1}{t_n(t_n+t_{n-1})}\le\frac{1}{q_n(q_n+q_{n-1})}\,.
\end{equation}
Lemma \ref{Cn_Bn_lem}\eqref{item:cn_gt_lem} implies that $c_k\geq b_k$ for all $1\leq k\leq n$. Therefore, from the recurrence law \eqref{eq:convergents} one can easily deduce that $t_k\geq q_k$ for every $1\leq k \leq n$. In particular, this holds for $k=n$ and $k=n-1$, which implies \eqref{tn_le_qn}. 

As for \eqref{eq:second}, fix any $b$ that satisfies $B_{n+1}(b)\cap C_n\neq\varnothing$. Setting $b_{n+1}=b$ and arguing as before, shows that \eqref{eq:second} can be rewritten as
\begin{equation*}
\label{qnq_qn_letn_tn1}
\frac{1}{q_{n+1}(q_{n+1}-q_{n})}<\frac{1}{t_n(t_n+t_{n-1})}\,.
\end{equation*}
The last inequality follows from Lemma \ref{Cn_Bn_lem}\eqref{item:bncn1_lem} with $n+1$ instead of $n$. Indeed, as $b_{n+1}\ge 2$, one has $q_{n+1}-q_n>q_n$ and therefore
\begin{equation*}
\label{qnq_qn_letn_tn1_2}
\frac{1}{q_{n+1}(q_{n+1}-q_{n})}<\frac{1}{q_{n+1} q_{n}}\,.
\end{equation*}
.
\end{proof}

Starting from this point and until the end of this section and through Sections \ref{sec:rationals} and \ref{sec:growth}, given a real number $\alpha$, the sequences $b_n(\alpha)$ and $c_n(\alpha)$ stand for the corresponding digits that arise from Shulga's algorithm with input $\alpha$. We will omit the dependency on $\alpha$ when it does not create ambiguity.

Now we are ready to prove Theorem \ref{Th1}.
\begin{proof}[Proof of Theorem 4]
Assume by contradiction that for some $n\geq2$ the partial quotients do not satisfy $c_n>b_n$. By \eqref{eq:main-shulga} this means that $c_n=b_n$. Denoting 
\begin{align*}
    x&=[0;b_{n-1},\ldots,b_1]\,,\\ 
    y&=[0;c_{n-1},\ldots,c_1]\,,\quad\textrm{and}\\
    z&=c_n=b_n\,,
\end{align*}
and substituting it into Lemma \ref{Cn_Bn_lem}\eqref{item:cn_gt_lem} and Lemma \ref{Cn_Bn_lem}\eqref{bn_gt_lem} gives
\begin{align*}
z&>(1+y)(z-1+x)\,,\quad\textrm{and}\\
z&>\left(1+\frac{1}{z}\right)\left(1+\frac{1}{y}\right)(1-x)-1\,,
\end{align*}
respectively. This is a contradiction since these inequalities can be rearranged into
\begin{align*}
z&<\left(1+\frac{1}{y}\right)(1-x)\,,\quad\textrm{and}\\
z&>\left(1+\frac{1}{y}\right)(1-x),
\end{align*}    
respectively.
\end{proof}

\section{Rational or divergent decomposition}\label{sec:rationals}
In this section we use Theorem \ref{Th1} in order to deduce Corollary \ref{cor:main}. Fix a number $\alpha\in[0,1]$ and let $b_n$ and $c_n$ be as defined in Defition \ref{def:shulga}. Also, let $p_n,q_n,s_n,t_n$ be as defined in \eqref{eq:fractions}.

\begin{lem}\label{lem:mon}
The sequence of ratios $\frac{t_n}{q_n}$ is strictly increasing.
\end{lem}

\begin{proof}
Indeed, since $b_n$ and $c_n$ are integers, the inequality \eqref{cnbn_main_ineq} implies that $c_n-b_n\ge 1$. Therefore, we have
\begin{equation}
\label{step_1_ratio}
\begin{aligned}
\frac{t_{n}}{t_{n-1}}\frac{q_{n-1}}{q_{n}}
&=
\frac{[c_{n};c_{n-1},\ldots,c_1]}{[b_{n};b_{n-1},\ldots,b_1]}
\\
&=1+\frac{(c_{n}-b_{n})+([0;c_{n-1},\ldots,c_1]-[0;b_{n-1},\ldots,b_1])}{b_{n}+[0;b_{n-1},\ldots,b_1]}>1+\frac{1}{2b_n+1}\,.
\end{aligned}
\end{equation}
The inequality in \eqref{step_1_ratio} follows from the trivial estimates 
$$
[0;c_{n-1},\ldots,c_1],\;[0;b_{n-1},\ldots,b_1] < 1/2\,,
$$
which follow from the fact that $c_{n-1},b_{n-1}\ge 2$. 
\end{proof}

\begin{lem}\label{lem:t_n_q_n_infty}
If the algorithm takes infinitely many steps then
\begin{equation}\label{eq:statement}
\lim\limits_{n\to\infty}\frac{t_n}{q_n}=\infty\,.
\end{equation}
\end{lem}
\begin{proof}
Suppose that the sequence $c_n$ is unbounded. From (\ref{CnBn_ineq1}) one can see that
\begin{equation*}
\frac{c_n}{2t^2_n}<\frac{1}{2t_{n-1}t_n}<\frac{2}{q^2_n}\,.
\end{equation*}
Thus,
\begin{equation*}
\frac{t^2_n}{q^2_n}>\frac{c_n}{4}\,.
\end{equation*}
Therefore, the divergence \eqref{eq:statement} follows directly from Lemma \ref{lem:mon}.

Considering the remaining case, if all the numbers $c_n$, and hence $b_n$, do not exceed some positive constant $N$, then \eqref{step_1_ratio} implies that
\begin{equation*}
\frac{t_n}{q_n}>\left(1+\frac{1}{2N-1}\right)^{n-1}\,,
\end{equation*}
thus finishing the proof.
\end{proof}

\begin{remark}
    In fact, it follows from the growth rate estimates of Section \ref{sec:growth} that the situation in the second paragraph of the proof of Lemma \ref{lem:t_n_q_n_infty} cannot occur, namely, that if $\{c_n\}_{n\geq1}$ is an infinite sequence then it is unbounded. However, since the above argument is short we decided to present it here and keep the more involved analysis for the next section.
\end{remark}

Now we are ready to prove Corollary \ref{cor:main}.
\begin{proof}[Proof of Corollary \ref{cor:main}]
First, we show that if $\alpha=p/q\in\mathbb{Q}$ then the algorithm converges in finitely many steps.
Let $C_n$ be the interval obtained on the $n$th step of the algorithm. The fact that $\alpha\in C_n$ is equivalent to the property
\begin{equation*}
\frac{p}{q}-\frac{p_n}{q_n}\in\bigl[[0;c_1,\ldots c_n],\ [0;c_1,\ldots,c_n+1]\bigr).    
\end{equation*}
Therefore, it follows from \eqref{cyl_def} that the denominator of $\frac{p}{q}-\frac{p_n}{q_n}$ is at least $t_n$. On the other hand, this denominator does not exceed $q_nq$. Thus, we have the inequality
\begin{equation*}
\frac{t_n}{q_n}\le q\,.    
\end{equation*}
Due to Lemma \ref{lem:t_n_q_n_infty}, the algorithm must stop at a finite step.

If $\alpha$ is irrational, one can easily see that the numbers $[0;b_1,b_2,\ldots]$ and $[0;c_1,c_2,\ldots]$ must lie in the set $G$. Indeed, from \eqref{c_n_bn1} we have
\begin{equation*}
q_{n+1}q_n>t_n^2\,.
\end{equation*}
Thus,
\begin{equation*}
b_{n+1}>\left(\frac{t_n}{q_n}\right)^2-1\,.    
\end{equation*}
Due to Lemma \ref{lem:t_n_q_n_infty}, the sequences $\{b_n\}_{n\geq1}$, and therefore $\{c_n\}_{n\geq1}$, both tend to infinity as $n$ tends to infinity.
\end{proof}

\section{Growth rate}\label{sec:growth}
In the previous section, we showed that the sequences $b_n$ and $c_n$ 
tend to infinity as $n\to\infty$. In this section, we provide some quantitative estimates on the growth rate of these sequences. First, we establish the inequality \eqref{monot_main_ineq}.
\begin{lem}
\label{bn_monot_lem}
For all $n\ge1$ one has $b_{n+1}\ge b_n$.    
\end{lem}
\begin{proof}
We suppose the contrary, and denote $b_n=t$. Then $c_n\ge t+1$ and $b_{n+1}\le t-1$. From \eqref{bn_ge_ineq_main}, applying the trivial estimates $1+\frac{1}{c_{n+1}}>1$ and $[0;c_{n-1},\ldots,c_1]>0$, we get
\begin{equation*}
t>\left(t+2\right)\left(1-\frac{1}{t}\right),
\end{equation*}
and thus obtain a contradiction, as $t\ge 2$.
\end{proof}
\begin{lem}
\label{lem_32}
Suppose that $b_n\ge 8$. If $b_{n+1}<c_n$, then $c_{n+1}>\frac{3}{2}b_n$.    
\end{lem}
\begin{proof}
Denote $c_n=t$, hence $b_{n+1}\le t-1$. Applying \eqref{bn_ge_ineq_main}, we obtain
\begin{equation*}
t>\left(1+\frac{1}{c_{n+1}}\right)\left(t+1\right)\left(1-\frac{1}{b_n}\right).
\end{equation*}
Thus, we have
\begin{equation*}
1>\left(1+\frac{1}{t}+\frac{1}{c_{n+1}}\right)\left(1-\frac{1}{b_n}\right),
\end{equation*}
which is equivalent to
\begin{equation}
\label{bn_ge_ineq_bn1_gt_cn2}
\frac{c_{n+1}}{b_{n}-1}>1+\frac{c_{n+1}}{t}\,.
\end{equation}
If $c_{n+1}\ge t$ then
\begin{equation*}
\frac{c_{n+1}}{b_n}>2\frac{b_n-1}{b_n}\ge\frac{7}{4}\,.
\end{equation*}
If $c_{n+1}<t$ then it follows from \eqref{bn_ge_ineq_bn1_gt_cn2} that
\begin{equation*}
c_n=t> 2(b_n-1)\,.
\end{equation*}
From \eqref{bn_ge_ineq_main} with $n+1$ instead of $n$ and Theorem \ref{Th1}, it follows that
\begin{align*}
c_{n+1}&\ge b_{n+1}+1>\left(1+\frac{1}{c_{n+1}}\right)\left(c_n+1+[0;c_{n+1},\ldots,c_1]\right)\left(1-[0;b_{n},\ldots,b_1]\right)
\\[2pt]
& >
\left(2(b_n-1)+1\right)\left(1-\frac{1}{b_n}\right)>2(b_n-1)\left(1-\frac{1}{b_n}\right)
\\[2pt]
&=
2b_n\left(1-\frac{1}{b_n}\right)^2\ge\frac{49}{32}b_n\,.    
\end{align*}
The lemma is proved.
\end{proof}

Now we are ready to prove Theorem \ref{Th2}.
\begin{proof}[Proof of Theorem \ref{Th2}]
The first inequality \eqref{monot_main_ineq} is already established in Lemma \ref{bn_monot_lem}. One can easily check (e.g., by a computer program) that \eqref{monot_main_ineq2} and \eqref{bn_main_ineq} hold if $n\le 6$. In fact, all $\alpha\in[0,1]$ for which $b_6$ is defined satisfy
\begin{equation}\label{eq:b_6}
    b_6\ge 8\,.
\end{equation} 
Note that for all $n,k\in\mathbb{N}$ we have $b_{n+k}\ge b_n$, $c_{n+k}>b_n$, so for $n\geq6$ we may apply the consequence of Lemma \ref{lem_32}. 

We start with the second inequality \eqref{monot_main_ineq2}. Denote $b_n=t$. Then $c_n\ge t+1$. If $b_{n+1}\ge t+1$ then $b_{n+2}\ge t+1=b_n+1$ by Lemma \ref{bn_monot_lem}, so \eqref{monot_main_ineq2} is established. Otherwise, by Lemma \ref{lem_32} we have $c_{n+1}\ge\frac{3}{2}b_n$, and therefore, applying \eqref{bn_ge_ineq_main} for index $n+2$ instead of $n$, we obtain
\begin{equation*}
b_{n+2}>\frac{3}{2}b_n\left(1-\frac{1}{b_{n+1}}\right)-1\ge \frac{3}{2}b_n\left(1-\frac{1}{b_{n}}\right)-1=\frac{3}{2}(b_n-1)-1>b_n+1\,. 
\end{equation*}
Thus, we established \eqref{monot_main_ineq2}.

Let us now prove the last inequality \eqref{bn_main_ineq}. Let $n>6$ be an arbitrary integer greater than $6$. If $b_{k+1}\ge c_k$ for all $6\leq k\leq n-2$ then Theorem \ref{Th1} implies by induction that $b_{n-1}\ge n+1$. Therefore, Lemma \ref{bn_monot_lem} gives $b_n \geq n+1$. Otherwise, let $l$ be the smallest index $6\leq l\le n-2$ satisfying $b_{l+1}<c_l$. Due to the definition of the index $l$ we conclude as above that $b_{l+1}\ge b_{l}\geq l+1$. By Lemma \ref{lem_32}, one has

\begin{equation*}
c_{l+1}>\frac{3}{2}b_l\,,
\end{equation*}
and therefore, applying \eqref{bn_ge_ineq_main} once again we have
\begin{equation*}
b_{l+2}>\frac{3}{2}b_l\left(1-\frac{1}{b_{l+1}}\right)-1\ge
\frac{3}{2}b_l\left(1-\frac{1}{b_{l}}\right)-1=\frac{3}{2}(b_l-1)-1\ge b_l+\frac{3}{2}\ge l+\frac{7}{2}\,.
\end{equation*}
Since $b_{l+2}$ and $l$ are integers this implies that $b_{l+2}\ge l+4=(l+2)+2$. This argument can be repeated with the index $6$ replaced by $l+2$, so it follows by induction with base \eqref{eq:b_6} that \eqref{bn_main_ineq} holds for every $n\geq1$. Thus, the proof of Theorem \ref{Th2} is completed.
\end{proof}
    
\subsection{Proof of Theorem \ref{examp_thm}}
In this subsection, we construct an irrational number $\alpha$ that satisfies Theorem \ref{examp_thm}. In our construction, we define the infinite sequences $b_1, b_2,\ldots$ and $c_1,c_2,\ldots$ and put
\begin{equation}
\label{alpha_def}
\alpha=[0;b_1,b_2,\ldots]+[0;c_1,c_2,\ldots]\,.
\end{equation}
Using \eqref{cn_bn_def}, we will define the intervals $B_1, C_1, B_2, C_2,\ldots$. Our goal will be to show that these intervals are nesting, i.e.,
\begin{equation}
\label{nesting_property}
B_1\supseteq C_1\supseteq B_2\supseteq C_2\supseteq\ldots\supseteq B_n\supseteq C_n\supseteq\ldots
\end{equation}
By Lemma \ref{large_intersec_lem}, the property \eqref{nesting_property} yields that \eqref{alpha_def} is the Shulga decomposition of $\alpha$. 

We put $b_1=2$ and $c_1=4$, and continue by recursion. Suppose that the numbers $b_1,\ldots,b_n$ and $c_1,\ldots,c_n$ are defined. Define 
\begin{equation}\label{eq:the_definition0}
    b_{n+1}\defeq c_n+2\,.
\end{equation}
The definition of $c_{n+1}$ is slightly more complicated. We define for $m=2,3$ the quantities
\begin{equation}\label{eq:the_definition}
c^{(m)}_{n+1}\defeq b_{n+1}+m\,,\quad \frac{s^{(m)}_{n+1}}{t^{(m)}_{n+1}}\defeq[0;c_1,\ldots,c_n,b_{n+1}+m]\,,
\end{equation}
and then
\begin{equation}\label{eq:the_definition2}
c_{n+1}\defeq 
\begin{cases}
    c^{(2)}_{n+1}&\textrm{if }0<c^{(2)}_{n+1}-\left(\frac{t^{(2)}_{n+1}}{q_{n+1}}\right)^2<1\,,\\
    c^{(3)}_{n+1}&\textrm{otherwise.}
\end{cases}
\end{equation}
So the first several digits are
\begin{equation}\label{eq:beginning}
\begin{aligned}
\relax
[0;b_1,\ldots,b_6,\ldots]&=[0;2,6,11,16,21,26,\ldots]\,,\\
[0;c_1,\ldots,c_6,\ldots]&=
[0;4,9,14,19,24,28,\ldots]\,.
\end{aligned}
\end{equation}

The proof of the nesting property \eqref{nesting_property} is based on the following observation: 


\begin{lem}
\label{cn_cn1_2}
For every $n\geq2$ the partial quotient $c_n$ satisfies
\begin{equation}
\label{tn_qn_ratio}
0<c_{n}-\left(\frac{t_{n}}{q_{n}}\right)^2 < 1\,.
\end{equation}
\end{lem}

\begin{proof}
The proof goes by induction. For $2\leq n \leq6$ the inequality \eqref{tn_qn_ratio} is verified by a direct computation. Let $n\geq6$ and
\begin{align*}
x\defeq c_{n}-\left(\frac{t_{n}}{q_{n}}\right)^2,
&&y\defeq c^{(2)}_{n+1}-\left(\frac{t^{(2)}_{n+1}}{q_{n+1}}\right)^2,
&&z\defeq c^{(3)}_{n+1}-\left(\frac{t^{(3)}_{n+1}}{q_{n+1}}\right)^2.
\end{align*}
We note that by definition \eqref{eq:the_definition} and by \eqref{eq:common}, one gets
\begin{equation}
\label{tn1_qn1_m_ratio_term_n}
\begin{aligned}
\left(\frac{t^{(2)}_{n+1}}{q_{n+1}}\right)^2&=\left(\frac{t_{n}}{q_{n}}\right)^2
\left(\frac{c_n+4+[0;c_n,\ldots,c_1]}{c_n+2+[0;b_n,\ldots,b_1]}\right)^2\\
&=
\left(1+\frac{2+[0;c_n,\ldots,c_1]-[0;b_n,\ldots,b_1]}{c_n+2+[0;b_n,\ldots,b_1]}\right)^2(c_n-x)\,.
\end{aligned}
\end{equation}
Due to $c_n>b_n$, one has $[0;c_n,\ldots,c_1]<[0;b_n,\ldots,b_1]$. Thus,
\begin{equation*}
\left(\frac{t^{(2)}_{n+1}}{q_{n+1}}\right)^2<\left(1+\frac{2}{c_n+2}\right)^2(c_n-x)=\left(\frac{c_{n}+4}{c_{n}+2}\right)^2(c_n-x)\,.
\end{equation*}
Thus, we obtain the following lower bound on $y$:
\begin{align*}
y&>c_n+4-\left(\frac{c_{n}+4}{c_{n}+2}\right)^2(c_n-x)=(c_n+4)\left(1-\frac{c_n+4}{c_n+2}\frac{c_n-x}{c_n+2}\right)\\[2pt]
&>(c_n+4)\left(1-\frac{c_n-x}{c_n}\right)=\frac{c_{n}+4}{c_n}x\,.
\end{align*}
Therefore, by definition \eqref{eq:the_definition2}, in order to complete the induction it is enough to show that if $0<x<1$ and $y\geq1$ then $0<z<1$. On one hand, the assumption $y\geq1$ and \eqref{eq:common} imply that
\begin{equation}\label{eq:similar}
\begin{aligned}
\left(\frac{t^{(3)}_{n+1}}{q_{n+1}}\right)^2
&=    
\left(\frac{t^{(2)}_{n+1}}{q_{n+1}}\right)^2\left(\frac{t^{(2)}_{n+1}+t_n}{t^{(2)}_{n+1}}\right)^2
\\[2pt]
&=
\left(\frac{t^{(2)}_{n+1}}{q_{n+1}}\right)^2\left(1+\left[0;c^{(2)}_{n+1},c_n,\ldots,c_1\right]\right)^2
\\
&< \left(c^{(2)}_{n+1}-1\right)\left(1+\frac{1}{c^{(2)}_{n+1}}\right)^2
\\[2pt]
&<c^{(2)}_{n+1}+1=c^{(3)}_{n+1}\,.
\end{aligned}
\end{equation}
So, $z>0$. On the other hand, since $n\geq6$, equation \eqref{eq:beginning} gives that $b_n,c_n\geq10$.
From \eqref{tn1_qn1_m_ratio_term_n} one gets
\begin{equation}
\label{tn1_qn1_m_ratio_term_ge}
\left(\frac{t^{(2)}_{n+1}}{q_{n+1}}\right)^2>\left(1+\frac{1.9}{c_n+2.1}\right)^2(c_n-x)=\left(\frac{c_{n}+4}{c_{n}+2.1}\right)^2(c_n-x)\,.
\end{equation}
Thus, we obtain the following upper bound on $y$:
\begin{align*}
y
&<c_n+4-\left(1+\frac{1.9}{c_{n}+2.1}\right)^2(c_n-x)
\\[2pt]
&<
c_n+4-\left(1+\frac{3.8}{c_{n}+2.1}\right)(c_n-x)
\\[2pt]
&=
4-\frac{3.8(c_n-x)}{c_n+2.1}(c_n-x)+x < 3/2\,,
\end{align*}
where the last inequality follows from $x<1$ and $c_n\geq26$. Therefore, similarly to \eqref{eq:similar}, one gets
\begin{align*}
\left(\frac{t^{(3)}_{n+1}}{q_{n+1}}\right)^2    
&=
\left(\frac{t^{(2)}_{n+1}}{q_{n+1}}\right)^2\left(\frac{t^{(2)}_{n+1}+t_n}{t^{(2)}_{n+1}}\right)^2
\\[2pt]
&>
\left(c^{(2)}_{n+1}-3/2\right)\left(1+\frac{1}{c^{(2)}_{n+1}+1}\right)^2
\\[2pt]
&>
\frac{\left(c^{(2)}_{n+1}-3/2\right)\left(c^{(2)}_{n+1}+3\right)}{c^{(2)}_{n+1}+1}
\\[2pt]
&>c^{(2)}_{n+1}=c^{(3)}_{n+1}-1\,,
\end{align*}
where the last inequality follows from $c_n\ge 10$. Thus, we established that $z<1$ and therefore the inequality \eqref{tn_qn_ratio} holds for all $n\geq1$.
\end{proof}
Now we are ready to prove Theorem \ref{examp_thm}.

\begin{proof}[Proof of Theorem \ref{examp_thm}]
Due to the definition of $b_n$ and $c_n$ in \eqref{eq:the_definition0} and \eqref{eq:the_definition2}, one can see that the inequality \eqref{2n_3n} is satisfied. Hence, it suffices to establish the nesting property \eqref{nesting_property}. 
First, let us show that $B_n\supseteq C_n$. 

With out loss of generality, we assume that $n$ is even. This containment is equivalent to (recall Figure \ref{fig:intervals})
\begin{equation}
\label{BnCn_compl_case}
[0;b_1,\ldots,b_{n-1},b_n]+[0;c_1,\ldots,c_{n-1},c_n]>[0;b_1,\ldots,b_{n-1},b_n-1]+[0;c_1,\ldots,c_{n-1}].
\end{equation}
The inequality \eqref{BnCn_compl_case} is equivalent to 
\begin{equation*}
\frac{1}{t_n t_{n-1}}<\frac{1}{q_n(q_n-q_{n-1})}\,,
\end{equation*}
which can be rewritten as
\begin{equation*}
c_n+[0;c_{n-1},\ldots,c_1]<\left(\frac{t_n}{q_n}\right)^2\frac{1}{1-[0;b_n,\ldots,b_1]}\,.
\end{equation*}
Applying \eqref{tn_qn_ratio}, one can see that it suffices to show
\begin{equation*}
c_n+\frac{1}{c_{n-1}}<\frac{c_n-1}{1-[0;b_n,\ldots,b_1]}\,,   
\end{equation*}
which is equivalent to
\begin{equation*}
1-[0;b_n,\ldots,b_1]<1-\frac{1+1/c_{n-1}}{c_n+1/c_{n-1}}\,.
\end{equation*}
Taking into account that $b_{n-1}\ge 2$, we see that it is sufficient to show
\begin{equation}
\label{ineq22bn1}
\frac{c_{n-1}+1}{c_nc_{n-1}+1}<\frac{2}{2b_n+1}\,.
\end{equation}
As $c_n\ge b_n+2=c_{n-1}+4$, the inequality \eqref{ineq22bn1} follows from
\begin{equation}
\label{ineq22bn1case3}
\frac{c_{n-1}+1}{c^2_{n-1}+4c_{n-1}+1}<\frac{2}{2c_{n-1}+5}\,.
\end{equation}
One can easily see that \eqref{ineq22bn1case3} holds whenever $c_{n-1}>3$. Thus, we established that $B_n\supseteq C_n$.

Now we show that $C_n\supseteq B_{n+1}$. 
This containment is equivalent to
\begin{equation}
\label{CnBn1_compl_case}
[0;b_1,\ldots,b_n,b_{n+1}-1]+[0;c_1,\ldots,c_n]<[0;b_1,\ldots,b_n]+[0;c_1,\ldots,c_n+1]\,.
\end{equation}
The inequality \eqref{CnBn1_compl_case} is equivalent to
\begin{equation}
\frac{1}{q_n(q_{n+1}-q_n)}<\frac{1}{t_n(t_n+t_{n-1})}   
\end{equation}
which can be written as
\begin{equation*}
\frac{1}{b_{n+1}-1+[0;b_n,\ldots,b_1]}\left(\frac{t_n}{q_n}\right)^2<\frac{1}{1+[0;c_n,\ldots,c_1]}\,.
\end{equation*}
Taking into account \eqref{tn_qn_ratio}, we obtain that it is sufficient to show that
\begin{equation}
\label{tnqn_cn_ineq}
\frac{c_n}{c_n+1+[0;b_n,\ldots,b_1]}<\frac{1}{1+[0;c_n,\ldots,c_1]}\,.
\end{equation}
Note that the left-hand side of \eqref{tnqn_cn_ineq} is less than $\frac{c_n}{c_n+1}$, while the right-hand side of \eqref{tnqn_cn_ineq} is greater than this quantity.
Thus, we established the property $C_n\supseteq B_{n+1}$ and Theorem \ref{examp_thm} is proved. 
\end{proof}

\bibliographystyle{plain}
\bibliography{refs_arx_ver.bib}

\end{document}